\documentclass[a4paper]{article}

\usepackage{geometry}

\usepackage{palatino}
\usepackage{latexsym}

\usepackage[mathscr]{eucal}
\usepackage{amsmath}
\usepackage{amsthm}
\usepackage{amsfonts}
\usepackage{amssymb}
\usepackage{amscd}
\usepackage{color}
\usepackage{graphicx}
\usepackage{graphics}
\usepackage{pifont}
\usepackage{subfigure}
\usepackage[makeroom]{cancel}
\usepackage[normalem]{ulem}
\usepackage[dvipsnames]{xcolor}
\usepackage{tikz}

\theoremstyle{plain}
\newtheorem{theorem}{Theorem}
\newtheorem{remark}[theorem]{Remark}

\newtheorem{proposition}[theorem]{Proposition}
\newtheorem{definition}[theorem]{Definition}

\newcommand{\ii}{\mathrm{i}}

\title{
Single-input perturbative control of a quantum symmetric rotor}

\begin{document}
\author{Thomas Chambrion and Eugenio Pozzoli\footnote{The authors are with Institut de Mathématiques de Bourgogne,
 UMR 5584,
 CNRS, Université Bourgogne Franche-Comté, F-21000 Dijon, France. E-mails:
     thomas.chambrion@u-bourgogne.fr, eugenio.pozzoli@u-bourgogne.fr.}}
\maketitle
\begin{abstract}
We consider the Schr\"odinger partial differential equation of a rotating symmetric rigid molecule (symmetric rotor) driven by a z-linearly polarized electric field, as prototype of degenerate infinite-dimensional bilinear control system. By introducing an abstract perturbative criterium, we classify its simultaneous approximate controllability; based on this insight, we numerically perform an orientational selective transfer of rotational population.
\end{abstract}

\section{Introduction}

\subsection{Physical model}
The attitude of a rigid body is a point in the Lie group of 
rotations ${\rm SO}(3)$, parametrized by the Euler's angles $(\alpha,\beta,\gamma)\in[0,2\pi)
\times[0,\pi]\times [0,2\pi)$.  At the quantum level, the state of the system is decribed by 
 the so-called \emph{wave function} $\psi:SO(3)\to \mathbf{C}$ whose square modulus $|\psi|^2$ can be 
 interpreted as a probability density. Throughout the paper, we use the Haar volume of $SO(3)$, $
 {\rm vol}_{\rm Haar}=\frac{1}{8}d\alpha d\gamma \sin(\beta) d\beta$ in Euler coordinates, as 
 reference measure without further notice and we require that $\psi$ belongs to the unit sphere of  the set $L^2({\rm 
 SO}(3))$ of  square integrable (for the Haar volume) complex functions on $SO(3)$, equipped with its natural $L^2$ norm. 
 
 When submitted to an external  $z$-linearly polarized electric field of (variable) real intensity $u$, the 
 dynamics of the wave function $\psi$ is given by the bilinear Schrödinger equation  
\begin{equation}\label{eq:top}
\ii\dot{\psi}=(H_{\rm rot}+u H_z)\psi, \quad \psi\in L^2({\rm SO}(3)),
\end{equation}
where  $H_z=-\delta \cos(\beta)$ is the interaction Hamiltonian between the $z$-polarization of 
the electric field and the electric dipole moment $\delta>0$ along the symmetry axis, 
\begin{align}\label{EQ_def_Hrot}
H_{\rm rot}&=-2A\left[\frac{1}{\sin(\beta)}\frac{\partial}{\partial \beta}\left(\sin(\beta)
\frac{\partial}{\partial \beta}\right)+\frac{1}{\sin^2(\beta)}\left(\frac{\partial^2}{\partial 
\alpha^2}  +\frac{\partial^2}{\partial \gamma^2} -2\cos(\beta)\frac{\partial^2}{\partial 
\alpha\partial \gamma}\right)\right]\nonumber \\ &+ \left(A-C\right)\frac{\partial^2}{\partial \gamma^2}, 
\end{align}
is the (essentially self-adjoint) rotational Hamiltonian, and $A,C>0$ are the rotational constants. 

Since the linear operator $H_z:L^2({\rm SO}(3))\to L^2({\rm SO}(3)) $ is bounded, standard arguments (see for instance \cite[Theorem 2.5]{BMS}) guarantee the well-posedness of \eqref{eq:top} for every locally integrable control function $u$. We denote with $(u,t)\mapsto \Upsilon^u_{t}$ the propagator at time $t$ of \eqref{eq:top}, i.e., for every $t$, the solution $\psi(t)$ at time $t$ of \eqref{eq:top} satisfies $\psi(t)=\Upsilon^u_t(\psi(0))$. 

An important question 
is the controllability of the above system \eqref{eq:top}, that is the possibility to chose a suitable 
(time variable) $u:[0,T]\to \mathbf{R}$ that drives the system from a known given state $
\psi(0)=\psi_0$ to (or close enough to) a given target $\psi(T)=\psi_{1}$.

\subsection{Contribution and main results}
The contribution of this paper is a characterization of the 
approximate controllability of the system \eqref{eq:top}. Our main result is the following.
\begin{theorem}\label{THE_main}
\begin{itemize}
\item[(i)] There exists a countable family $(\mathcal{H}_{n})_{n\in \mathbf{N}}$ of orthogonal closed infinite dimensional subspaces of $L^2({\rm SO}(3))$ such that, for every $T>0$, for every $u$ in $L^1([0,T],\mathbf{R})$, for every $n$ in $\mathbf{N}$, $\Upsilon^u_{T} (\mathcal{H}_n) \subset \mathcal{H}_n$. Moreover, also the orthogonal complement $\mathcal{G}$ in $L^2({\rm SO}(3))$ of 
$\mathcal{H}:=\displaystyle {{\bigoplus}_{n\in \mathbf{N}}} \mathcal{H}_{n}$ is invariant for the propagators of \eqref{eq:top}, and the dynamics in $\mathcal{G}$ are completely determined by the dynamics in $\mathcal{H}$.

\item[(ii)] Denoting with $p_n$ the orthogonal projection of $L^2({\rm SO}(3))$ onto $\mathcal{H}_n$, for every $\epsilon>0$, and every $\psi_0$, $\psi_1$ in $\mathcal{H}$
such that $\|p_n(\psi_0)\|=\|p_n(\psi_1)\|$ for all $n\in\mathbf{N}$, there exist $T>0$ and $u\in L^1([0,T],\mathbf{R})$ such that $\|\Upsilon^u_T(\psi_0)-\psi_1 \|<\epsilon$. 
\end{itemize}
\end{theorem}

 The first statement is indeed both an 
obstruction to controllability, since the norm of each $\mathcal{H}_n$ component of the wave function is 
conserved for any choice of control, and a partial obstruction to simultaneous controllability, since the dynamics in $\mathcal{G}$ are related to the dynamics in $\mathcal{H}$ for any choice of control.
The second part of Theorem \ref{THE_main} states a simultaneous approximate controllability result in $\mathcal{H}$ w.r.t. $n$, and may be refined in the following way.
\begin{proposition}\label{PRO_compelment}
 In the second statement of Theorem 
\ref{THE_main}, $u$ can be chosen to be analytic instead of $L^1$ and the  
majoration $\|\Upsilon^u_T(\psi_0)-\psi_1 \|<\epsilon$ can be required to hold for the graph norm of $H_{\rm rot}^k$ for any $k$ in $\mathbf{N}$. 
\end{proposition}  

 Beside this theoretical result, we show with a numerical example that the proof is constructive, as it furnishes a method to obtain explicit control laws inducing a selective transfer 
between eigenstates of the rotational Hamiltonian.

\subsection{A brief survey of the literature}
The study of the controllability properties of quantum systems modelled through the bilinear Schr\"odinger equation is a fundamental problem for applications in physics and chemistry. Molecular systems are prototypes of degenerate systems and have been investigated in theoretical physics since the early days of quantum control \cite{rabitz,rama,Schirmer_2002,turinici-rabitz}, with well-established experimental applications in quantum chemistry \cite{PattersonNature13} and recently theoretical ones in quantum computation \cite{victor}. 
For an overview on the controllability of molecular rotation and its applications we refer also to \cite{koch}. In an abstract framework, one usually writes the dynamics as  
\begin{equation}\label{eq:intro}
\ii \dot{\psi}=(H_0+uH_1)\psi, 
\end{equation}
where $H_0$  and $H_1$ are self-adjoint operators on some Hilbert space, endowed with Hilbert product $\langle \cdot, \cdot \rangle$. When the Hilbert space is finite-dimensional, controllability is well-understood in terms of Lie-algebraic conditions \cite{lierank,Schirmer2}. When the dimension is infinite, the question of the controllability of such quantum bilinear control systems raised much interest in 
the last two decades, and has been attacked with various techniques  (see, e.g., \cite{Coron, laurent} for fixed point techniques, \cite{Mirra2,nersesyan} for Lyapunov techniques or \cite{bloch, BCMS, Keyl} for the geometric techniques similar to our approach in this work).

\subsubsection{Obstruction to (simultaneous) controllability}
An obvious obstruction to the controllability of system \eqref{eq:intro} is the stability of strict closed Hilbert subspaces by $H_0$ and $H_1$. This situation has already been noted in \cite{Ugo-Mario-Io-symmetrictop}. A less obvious obstruction is the existence of isomorphisms between decoupled dynamics that makes them related, hence not simultaneously controllable: this is the content of the second statement of Theorem \ref{THE_main}(i).

\subsubsection{Averaging and selective excitation}
Averaging is a standard technique to induce a rotation on the subspace spanned by two eigensates $\phi_1$ and $\phi_2$ of $H_0$, associated with simple eigenvalues $\lambda_1$ and $\lambda_2$ by using a periodic control with period $\dfrac{2\pi}{|\lambda_1-\lambda_2|}$ \cite{chambrion}. 
Under generic conditions, this technique is extremely efficient in large time.   
The main difficulty in controlling degenerate quantum systems 
is that a periodic control pulse that oscillates in resonance with a spectral gap 
$ |\lambda_1-\lambda_2|$ of the drift does not select in general only one transition between two corresponding eigenstates, as it excites transitions between all couples of eigenstates each belonging to one of the two addressed degenerate eigenspaces. To overcome this difficulty, we use a  perturbative approach (as in \cite{BCMS, panati, duca, zhang}), replacing $u$ by $u(t)=\mu+v(t)$ 
for a suitable constant $\mu$, and taking $v$ to be periodic in resonance with the spectral gaps of $H_0+\mu H_1$. It is interesting to notice that the idea of perturbing the rotational spectrum with an electric field to lift the degeneracies has a long history in spectroscopy experiments \cite{dakin}.

 Alternatively, the procedure of breaking coupled rotational transitions is often conducted by physicists by means of several orthogonal controls (so-called multi-polarization, see e.g. \cite{turinici-rabitz-linear-top, monika_eugenio} for controllability results on finite dimensional modal truncations with three orthogonal control fields). The controllability of the corresponding PDEs (with three orthogonal control fields) has been established in \cite{BCS,Ugo-Mario-Io-symmetrictop,
asymm-top}.




\subsubsection{Novelty of the contribution} 

 This work is the first one dealing with the controllability of the orientation of the symmetric molecule with one control field only. We give a complete description of the approximate controllability properties of this system. This settles an open question asked in Section II-E of \cite{koch}. The techniques we use are proved effective with a numerical example.

\subsection{Content of the paper}
Section \ref{SEC_Simultaneous_approximate_main_section} is devoted to the proof of Theorem \ref{THE_main}. A general abstract controllability test in presented in Section \ref{sec:simcontrol}, and applied to our example in Section \ref{sec:molecule}. Section \ref{sec:simulations} presents the result of numerical simulations. 



\section{Simultaneous approximate controllability : a perturbative approach}
\label{SEC_Simultaneous_approximate_main_section}

\subsection{Non-resonant chains of connectedness}
\begin{definition}\label{assumption1}
A couple of linear operators $(A,B)$ on an infinite-dimensional Hilbert space $\mathfrak{H}$ satisfies $\mathbb{A}$ if
\begin{itemize}
\item[(i)] $A$ (with domain $D(A)$) is a skew-adjoint unbounded operator, with discrete simple spectrum (that is, every point in the spectrum is a purely imaginary eigenvalue with multiplicity one);
\item[(ii)] $B$ is a skew-adjoint bounded operator.
\end{itemize}
\end{definition}
We consider the family of systems
\begin{equation}\label{eq:m}
 \dot{\psi}=(A_m+uB_m)\psi,\quad  \psi\in\mathcal{H}_m,
 \end{equation}
$m\in \mathbf{N}$, where $\mathcal{H}_m$ is an infinite-dimensional Hilbert space, $u\in L^1_{\rm loc}(\mathbf{R},\mathbf{R})$, and $(A_m,B_m)$ is supposed to satisfy $\mathbb{A}$ for every $m\in \mathbf{N}$. We denote by $\Upsilon^{u,m}_t$ the propagator of \eqref{eq:m} (and we shall drop the dependence of $m$ when it is applied to an initial datum in $\mathcal{H}_m$ since there is no ambiguity).
\begin{definition}
System \eqref{eq:m} is approximately controllable if for every $\psi_0,\psi_1$ in $\mathcal{H}_m$ with $\|\psi_0\|=\|\psi_1\|$, and every $\epsilon> 0$ there exist $T\geq 0$ and $u\in L^1([0,T],\mathbf{R})$ such that $\| \Upsilon_T^u(\psi_0)-\psi_1\| <  \epsilon. $
\end{definition}
We denote the sets of eigenvalues and eigenfunctions of $A_m$, resp., by $\Lambda_m:=\{\lambda_j^m\}_{j\in\mathbf{N}}$ and $\Phi_m:=\{\phi^m_j\}_{j\in\mathbf{N}}$,
we introduce the notation $b_{(j,m),(j',m)}:=\langle \phi_j^m,B_m \phi_{j'}^{m}\rangle$, and the set $\Xi_m:=\{(j,m)\}_{j\in \mathbf{N}}$ that labels the eigenfunctions of $A_m$.
\begin{definition}
The operator $B_m$ is said to be \emph{connected} w.r.t. $\Phi_m$ if for any couple of labels $\rho,\xi\in\Xi_m$ there exists a finite sequence $\{(\rho^1_1,\rho^1_2),\dots,(\rho^p_1,\rho^p_2)\}\subset \Xi_m^2$ that connects them, that is \begin{itemize}
\item $\rho_1^1=\rho$ and $\rho^p_2=\xi$;
\item $\rho^n_2=\rho^{n+1}_1, \forall n=1,\dots,p-1$;
\item $b_{\rho^n_1,\rho^n_2}\neq 0, \forall n=1,\dots ,p$.
\end{itemize}
\end{definition}
If $B_m$ is connected w.r.t. $\Phi_m$, one can choose a chain of connectedness $S_m\subset \Xi_m^2$ w.r.t. $\Phi_m$, that is a sequence that connects any couple of labels $\rho,\xi \in \Xi_m$.
\begin{definition}
A chain of connectedness $S_m$ is said to be non-resonant w.r.t. $\Phi_m$ if for any $(\rho=(j,m),\xi=(j',m))\in S_m$, one has $|\lambda^m_j-\lambda^m_{j'}|\neq |\lambda^m_l-\lambda^m_{l'}|$ for all $((l,m),(l',m))\in \Xi_m^2\setminus\{(\rho,\xi),(\xi,\rho)\}$ such that $b_{(l,m),(l',m)}\neq 0$.
\end{definition}
The following result is the starting point of our analysis.
\begin{theorem}[\cite{BCCS}]\label{thm:chain}
If \eqref{eq:m} admits a non-resonant chain of connectedness w.r.t. $\Phi_m$, then it is approximately controllable.
\end{theorem}
\subsection{A simultaneous approximate controllability test}\label{sec:simcontrol}
We remark that in \eqref{eq:m} the control function $u$ does not depend on $m$, meaning that our goal is to simultaneously control a family of systems with the same external field.
\begin{definition}
The family of systems \eqref{eq:m}, $m\in\mathbf{N}$, is \emph{simultaneously approximately controllable} if, for every $r\in \mathbf{N}$, every $\psi_{0}^{m_j},\psi_{1}^{m_j}\in\mathcal{H}_{m_j}$ with $\|\psi_{0}^{m_j}\|=\|\psi_{1}^{m_j}\|$, and every $\epsilon >0$, there exist $T\geq 0$ and $u\in L^1([0,T],\mathbf{R})$ such that
$\|\Upsilon^{u}_{T}(\psi_{0}^{m_j})-\psi_{1}^{m_j}  \| <\epsilon$, $\forall j=1,\dots,r.$
\end{definition}
Notice that the controllability of each single system does not imply in general the simultaneous controllability of the family; indeed, a part of it may not be simultaneously controllable with only one external field. This is the case, e.g., for the evolutions in $\mathcal{H}_m$ and $\mathcal{H}_{m'}$ if $\mathcal{H}_m=\mathcal{H}_{m'}$, $A_m=A_{m'}$ and $B_m=B_{m'}$ for some $m'\neq m$. When two spectral gaps corresponding to two different drifts $A_m$ and $A_{m'}$, $m\neq m'$, happen to be equal (that is, a spectral degeneracy appears in the family of systems), the variation of the eigenvalues of $A_m$ (resp. $A_{m'}$) under the action of $B_m$ (resp. $B_{m'}$), considered as a perturbation, can lift such degeneracy and thus furnish the simultaneous controllability in $m$ and $m'$. This is the content of the next result, where the variation is expanded up to the second order w.r.t. the perturbation parameter.
\begin{theorem}\label{thm:sim}
Suppose that for every $m\in\mathbf{N}$ system \eqref{eq:m} admits a non-resonant chain of connectedness $S_m$ w.r.t. $\Phi_m$ and either one of the following holds
\begin{itemize}
\item[(i)] $\lambda^m_j-\lambda^m_{j'}=\pm (\lambda^n_{l}-\lambda^n_{l'})$ for some $((j,m),(j',m))\in S_m$ and
$((l,n),(l',n))\in \Xi_n^2$ implies
\begin{align}\label{eq:first}
b_{(j,m),(j,m)}-b_{(j',m),(j',m)}  \neq \pm  \left( b_{(l,n),(l,n)}-b_{(l',n),(l',n)}\right);
\end{align}
\item[(ii)]$\lambda^m_j-\lambda^m_{j'}=\pm (\lambda^n_{l}-\lambda^n_{l'})$ and $
b_{(j,m),(j,m)}-b_{(j',m),(j',m)}=\pm (b_{(l,n),(l,n)}-b_{(l',n),(l',n)})$ for some $((j,m),(j',m))\in S_m$
and $((l,n),(l',n))\in \Xi_n^2$ implies
\begin{eqnarray}
\sum_{k\neq j}\frac{|b_{(j,m),(k,m)}|^2}{\lambda^m_k-\lambda^m_j}-\sum_{k\neq j'}\frac{|b_{(j',m),(k,m)}|^2}{\lambda^m_k-\lambda^m_{j'}}
\neq  \pm\left(\!\! \sum_{k\neq l}\frac{|b_{(l,n),(k,n)}|^2}{\lambda^n_k-\lambda^n_l} \!- \!\sum_{k\neq l'}\frac{|b_{(l',n),(k,n)}|^2}{\lambda^n_k-\lambda^n_{l'}}
\!\!\right)\!\!.\label{eq:second}
\end{eqnarray}
\end{itemize}
Then, the family of systems \eqref{eq:m}, $m\in\mathbf{N}$, is simultaneously approximately controllable.
\end{theorem}
\begin{proof}
\textbf{Step 1}: If there are no degenerate transitions, that is, if for all $m,n\in\mathbf{N}$
\begin{equation}\label{eq:noresonances}
|\lambda^m_j-\lambda^m_{j'}|\neq|\lambda^n_{l}-\lambda^n_{l'}| 
\end{equation}
 for all $((j,m),(j',m))\in S_m$ and all $((l,n),(l',n))\in \Xi_n^2$ such that $b_{(l,n),(l',n)}\neq 0$,
then the family of systems \eqref{eq:m}, $m\in\mathbf{N}$, is simultaneously approximate controllable. Indeed, by denoting for any $m,N\in\mathbf{N}$
$$\Sigma^{(N)}_m:=\{|\lambda^m_j-\lambda^m_{j'}|, ((j,m),(j',m))\in S_m, j,j'\leq N \}$$ the set of spectral gaps of the chain $S_m$ which connect states $\phi_j^m,\phi_{j'}^m$ with $j,j'\leq N$, one has that for any $r\in\mathbf{N}$, any $\sigma\in \Sigma^{(N)}_{m_1}$, and any $\tau,\epsilon>0$, there exists a control $u\in L^1([0,T],\mathbf{R})$ such that \cite[Prop. 4.1]{CS}
\begin{align}
 \|\Upsilon^{u,m_1}_T -e^{\tau \mathcal{E}_\sigma(B_{m_1})}\|_{\mathcal{L}(\mathcal{H}_{m_1}^{(N)},\mathcal{H}_{m_1}^{(N)})}&<\epsilon \label{eq:propagator1} \\
 \|\Upsilon^{u,m_j}_T-{\rm Id} \|_{\mathcal{L}(\mathcal{H}_{m_j}^{(N)},\mathcal{H}_{m_j}^{(N)})}&<\epsilon,\forall j=2,\dots,r \label{eq:propagator2}
 \end{align}
 where $\mathcal{H}_m^{(N)}:={\rm span}\{\phi_j^{m},j=1,\dots,N\}$, $\|M\|_{\mathcal{L}(\mathcal{H}_m^{(N)},\mathcal{H}_m^{(N)})}$ denotes the operator norm of any matrix $M:\mathcal{H}_m^{(N)}\rightarrow\mathcal{H}_m^{(N)}$ and the operator $\mathcal{E}_\sigma(B_{m})$ is defined for every $\sigma\geq 0$ as
 $$\langle \phi_j^m, \mathcal{E}_\sigma(B_{m})\phi_{j'}^m\rangle=\begin{cases}
 \langle \phi_j^m,B_m\phi_{j'}^m\rangle, & \text{ if } |\lambda_j^m-\lambda_{j'}^m|=\sigma\\
 0, & \text{ if } |\lambda_j^m-\lambda_{j'}^m|\neq\sigma.
 \end{cases} $$
 The existence of a control $u$ that verifies \eqref{eq:propagator1} and \eqref{eq:propagator2} is guaranteed by the fact that each $S_m$ is non-resonant w.r.t. $\Phi_m$ and by \eqref{eq:noresonances}. Then, \eqref{eq:propagator1} and \eqref{eq:propagator2} imply the simultaneous approximately controllability of the family of systems \eqref{eq:m}, $m\in\mathbf{N}$. \\
 \textbf{Step 2:} If there are resonant transitions but \eqref{eq:first} or \eqref{eq:second} are satisfied, we take a shifted control $u(t) = v(t)+\mu$, obtaining
 \begin{equation}\label{eq:simperturbed}
 \dot{\psi}=(A_m+\mu B_m)\psi+v B_m\psi, \quad \psi\in \mathcal{H}_m.
 \end{equation}
 for $\mu>0$. Then, being $B_m$ bounded, we have that \cite{kato}
  \begin{align*}
 \frac{d}{d\mu}\Big|_{\mu=0}\lambda^m_j(\mu) &= b_{(j,m),(k,m)},\\
 \frac{d^2}{d\mu^2}\Big|_{\mu=0}\lambda^m_j(\mu)&=\sum_{k\neq j}\frac{|b_{(j,m),(k,m)}|^2}{\lambda^m_k-\lambda^m_j},
 \end{align*}
  where $\{\lambda^m_j(\mu)\}_{j\in\mathbf{N}}$ are the eigenvalues (analytic w.r.t. $\mu$) of $A_m+\mu B_m$. A $2^{\rm nd}$ order Taylor expansion then shows that
$$ |\lambda^m_j(\mu)-\lambda^m_{j'}(\mu)|\neq|\lambda^n_{l}(\mu)-\lambda^n_{l'}(\mu)|,\quad \text{for a.e. }\mu, $$
 for all $((j,m),(j',m))\in S_m$ and all $((l,n),(l',n))\in \Xi_n^2$. 
Also, $\langle \phi_j^m(\mu),B_m \phi_{j'}^m(\mu)\rangle\neq 0$ for a.e. $\mu$ if $b_{(j,m),(j',m)}\neq 0$, where $\{\phi^m_j(\mu)\}_{j\in\mathbf{N}}$ are the eigenfunctions (analytic w.r.t. $\mu$) of $A_m+\mu B_m$. We can then apply Step 1 to the family of systems \eqref{eq:simperturbed}, $m\in\mathbf{N}$, by replacing any $|\lambda^m_j-\lambda^m_{j'}|\in\Sigma^{(N)}_m$ with the corresponding $|\lambda^m_j(\mu)-\lambda^m_{j'}(\mu)|$.
\end{proof}

\subsection{Proof of Theorem \ref{THE_main}}
\label{sec:molecule}
In this section we apply Theorem \ref{thm:sim} to the explicit physical system \eqref{eq:top}.
Since $H_{\rm rot}:H^2({\rm SO}(3))\rightarrow L^2({\rm SO}(3))$ is the Laplace-Beltrami operator of the compact manifold ${\rm SO}(3)$ (endowed with the diagonal Riemannian metric ${\rm diag}(A,A,C)$), it has discrete spectrum. The spectral decomposition of $H_{\rm rot}$ is explicit, given in terms of the Wigner $D$-functions $D_j^{k,m}(\alpha,\beta,\gamma)=e^{i(k\gamma+m\alpha)}d_j^{k,m}(\beta)$, $ j\in\mathbf{N}, k,m=-j,\dots,j$, where $d_j^{k,m}$ solves a suitable Legendre differential equation, and reads \cite{gordy}
\begin{equation}\label{eq:spectrum}
H_{\rm rot}D_j^{k,m}=(Aj(j+1)-(A-C)k^2)D_j^{k,m}=:E_j^{k,m}D_j^{k,m},
\end{equation}
for $j\in\mathbf{N},k,m=-j,\dots,j$. Equation \eqref{eq:spectrum} defines the eigenvalues $E_j^{k,m}$ of $H_{\rm rot}$: each $E_j^{k,m}$ has a $2$-dimensional degeneracy w.r.t. $k$, and a $(2j+1)$-dimensional degeneracy w.r.t. the angular momentum orientation $m$: the eigenspace of $E_j^{k,m}$ is thus given by
$\mathcal{E}_j^{ k}:={\rm span}\{D_j^{k,m},D_j^{-k,m}\}_{m=-j,\dots,j}.$
Thanks to the spectral theorem of unbounded self-adjoint operators, one has the orthonormal decomposition of the ambient Hilbert space $L^2({\rm SO}(3))=\overline{{\rm span}}\{D_j^{k,m}\}_{ j\in\mathbf{N},k,m=-j,\dots,j.} $.
The selection rules for $H_z$ w.r.t. the Wigner $D$-functions are \cite{gordy}
\begin{equation}\label{eq:rules}
\langle  D^{k,m}_j \!\!,  H_z D^{k',m'}_{j'} \rangle=0,\text{if } |j-j'|>1, \text{or } k\neq k'\!, \text{or } m\neq m'.
\end{equation}
The non-vanishing matrix elements of $H_z$ are \cite{gordy}
\begin{align}
\langle  D^{k,m}_j ,  \ii H_zD^{k,m}_{j} \rangle &=\ii \delta \dfrac{km}{j(j+1)}=:b_{(j,k,m),(j,k,m)}\label{eq:diagonal},\\ 
\langle  D^{k,m}_j , \ii H_zD^{k,m}_{j+1} \rangle &=\ii \delta  \dfrac{[(j+1)^2-k^2]^{1/2}[(j+1)^2\!-\!m^2]^{1/2}}{-(j+1)[(2j+1)(2j+3)]^{1/2}}\nonumber \\
&=:b_{(j,k,m),(j+1,k,m)}.\label{eq:updiagonal}
\end{align}
For any $(k,m)\in\mathbf{Z}^2$, we consider the infinite-dimensional closed subspace
$\mathcal{H}_{k,m}:=\overline{{\rm span}}\{D_j^{k,m}\mid j\in\mathbf{N}, j\geq \max\{|m|,|k|\}\}$ of $L^2({\rm SO}(3))$ and denote by $p_{(k,m)}:L^2({\rm SO}(3))\rightarrow\mathcal{H}_{k,m}$ the orthogonal projection. We notice that $(H_{\rm rot}|_{\mathcal{H}_{k,m}},H_z|_{\mathcal{H}_{k,m}})$ satisfies $\mathbb{A}$ for every $(k,m)\in\mathbf{Z}^2$. Since 
$\bigoplus_{(k,m)\in\mathbf{Z}^2}\mathcal{H}_{k,m}$ is dense in $L^2({\rm SO}(3))$ and 
each $\mathcal{H}_{k,m}$ is invariant for the propagators of \eqref{eq:top} (cf. \eqref{eq:rules}), system \eqref{eq:top} can be naturally seen as the family of systems
\begin{equation}\label{eq:km}
\ii \dot{\psi}=(H_{\rm rot}|_{\mathcal{H}_{k,m}}+uH_z|_{\mathcal{H}_{k,m}})\psi,\quad \psi\in\mathcal{H}_{k,m},
\end{equation}
$(k,m)\in\mathbf{Z}^2$. We define the set $\mathcal{N}:=\{(k,m)\in\mathbf{Z}\times\mathbf{N}\mid |k|\leq m\}$. The next result classifies which part of \eqref{eq:top} is simultaneously controllable (compare also with Fig. \ref{fig:graphs}).
\begin{theorem}\label{thm:top}
\begin{itemize}
\item[(a)]\textbf{Related dynamics:} Let $(k,m)\in\mathcal{N}$, then the linear isomorphisms defined on the basis as
\begin{align*}
f_{(k,m),1}:\mathcal{H}_{k,m}&\rightarrow \mathcal{H}_{-k,-m}\\
D_j^{k,m}&\mapsto D_j^{-k,-m}, \\
f(t)_{(k,m),2}:\mathcal{H}_{k,m}&\rightarrow \mathcal{H}_{m,k}\\
D_j^{k,m}&\mapsto e^{\ii t( A-C) (k^2-m^2)}D_j^{m,k},\\
f(t)_{(k,m),3}:\mathcal{H}_{k,m}&\rightarrow \mathcal{H}_{-m,-k}\\
D_j^{k,m}&\mapsto e^{\ii t(A-C) (k^2-m^2)}D_j^{-m,-k},
\end{align*}
are such that
\begin{equation}\label{eq:related}
\Upsilon^u_t\circ p_{(k,m)}=f(t)_{(k,m),i}^{-1}\circ\Upsilon^u_t\circ f(t)_{(k,m),i}\circ p_{(k,m)}
\end{equation}
for all $t\in\mathbf{R}$, all $i=1,2,3$, and all $u\in L^1_{\rm loc}(\mathbf{R},\mathbf{R})$.
\item[(b)]\textbf{Non-related dynamics:} The family of systems \eqref{eq:km},
$(k,m)\in\mathcal{N}$, is simultaneously approximately controllable.
\end{itemize}
\end{theorem}
\begin{proof}
In order to prove (a), we first notice that 
$$\langle D_j^{k,m}, H_z D_{j+h}^{k,m}\rangle=\langle f(t)_{(k,m),i}D_j^{k,m}, H_z f(t)_{(k,m),i}D_{j+h}^{k,m}\rangle$$ for all $h=0,1$, $j\geq \min\{|k|,m\}$, $t\in\mathbf{R}$ and $i=1,2,3$ (cf. \eqref{eq:diagonal} and \eqref{eq:updiagonal}). Also, 
$$\langle D_j^{k,m}, H_{\rm rot} D_{j}^{k,m}\rangle=\langle f_{(k,m),1}D_j^{k,m}, H_{\rm rot} f_{(k,m),1}D_{j}^{k,m}\rangle$$ for all $j\geq \min\{k,|m|\}$ (cf. \eqref{eq:spectrum}), which implies \eqref{eq:related} for $i=1$. Finally, for $i=2,3$, 
\begin{align*}
\langle D_j^{k,m}\!\!, H_{\rm rot} D_{j}^{k,m}\rangle &\!=\!\langle f(t)_{(k,m),i}D_j^{k,m}\!, H_{\rm rot} f(t)_{(k,m),i}D_{j}^{k,m}\rangle\\ &-(A-C)(k^2-m^2)
\end{align*}
 for all $j\geq \min\{|k|,m\}$ and $t\in\mathbf{R}$  (cf. \eqref{eq:spectrum}), which implies \eqref{eq:related} for $i=2,3$ and concludes the proof of (a).

The proof of part (b) is an application of Theorem \ref{thm:sim}: for any $(k,m)\in\mathcal{N}$ we consider the chain of connectedness $S_{(k,m)}:=\{((j,k,m),(j+1,k,m)),j\geq \min\{|k|,m\}\}$, which is non-resonant w.r.t. the eigebasis $\{D^j_{k,m}\mid j\geq \min\{k,|m|\}\}$ of $H_{\rm rot}|_{\mathcal{H}_{k,m}}$. Using \eqref{eq:spectrum} and \eqref{eq:rules}, we check the resonances w.r.t. the eigenbasis of $H_{\rm rot}|_{\mathcal{H}_{k',m'}}$ for $(k',m')\neq (k,m)$: since $E_{j+1}^{k,m}-E_{j}^{k,m}=2A(j+1)$, then
$$E_{j+1}^{k,m}-E_{j}^{k,m}=E_{j'+1}^{k',m'}-E_{j'}^{k',m'}$$
if and only if  $j'=j$ and $k',m'=-j,\dots,j.$
Since $b_{(j+1,k,m),(j+1,k,m)}-b_{(j,k,m),(j,k,m)}=-2\delta\frac{km}{j(j+1)(j+2)}$  (cf. \eqref{eq:diagonal}), then
\begin{align*}
 b_{(j+1,k,m),(j+1,k,m)}-b_{(j,k,m),(j,k,m)} = b_{(j+1,k',m'),(j+1,k',m')}-b_{(j,k',m'),(j,k',m')}
\end{align*}
if and only if $k'm'= km$. Hence, by applying Theorem \ref{thm:sim}(i), we conclude that the family of systems \eqref{eq:km} with $(k,m)\in\mathcal{N}$ and $km\neq k'm'$ is simultaneously approximately controllable. When $km=k'm'$, we consider the second order condition: thanks to \eqref{eq:rules}, this is equivalent to solve the equality
\begin{align*}
&\sum_{\pm}\!\frac{|b_{(j+1,k,m),(j+1\pm 1,k,m)}|^2}{E^{0,m}_{j+1\pm1}-E^{0,m}_{j+1}}\!-\!\sum_{\pm}\!\frac{|b_{(j,k,m),(j\pm1,k,m)}|^2}{E^{0,m}_{j\pm 1}-E^{0,m}_j}\\ = &\sum_{\pm}\!\frac{|b_{(j+1,k,m'),(j+1\pm 1,k,m')}|^2}{E^{0,m'}_{j+1\pm1}-E^{0,m'}_{j+1}}\!-\!\sum_{\pm}\!\frac{|b_{(j,k,m'),(j\pm1,k,m')}|^2}{E^{0,m'}_{j\pm 1}-E^{0,m'}_j},
\end{align*}
which reads
$Q(j)(m'^2+k'^2-m^2-k^2)=0$ (cf. \eqref{eq:updiagonal}), where $Q(j)$ is a quotient of polynomials in $j$ that has no positive integer zeros nor poles,
which implies $(k',m')=(k,m)$, under the assumptions $km=k'm'$, $(k,m),(k',m')\in\mathcal{N}$. By applying Theorem \ref{thm:sim}(ii), we conclude that the family of systems \eqref{eq:km} with $(k,m)\in\mathcal{N}$ is simultaneously approximately controllable.
\end{proof}
\begin{remark}
By noticing that $D_j^{0,m}(\alpha,\beta,\gamma)=Y_j^m(\alpha,\beta)$, where $Y_j^m$ are the spherical harmonics, that is, the eigenfunctions of the Laplace-Beltrami operator $\Delta_{S^2}$ of the $2$-sphere $S^2\subset \mathbf{R}^3$, one has that $L^2(S^2)=\overline{\rm span}\{D_j^{0,m}\mid j\in\mathbf{N}, m=-j\dots,j\}$ and $H_{\rm rot}|_{\bigoplus_{m\in\mathbf{Z}}\mathcal{H}_{0,m}}=H_{\rm rot}|_{L^2(S^2)}=-2A\Delta_{S^2}$. Hence, the case $k=0$ in Theorem \ref{thm:top} classifies the simultaneous approximate controllability w.r.t. the orientational quantum number $m$ of the Schr\"odinger equation
$$\ii \dot{\psi}=(-\Delta_{S^2}-u\,\delta\cos(\beta))\psi,\quad\psi\in L^2(S^2).$$
of a rotating linear molecule (compare also with Fig. \ref{fig:graphs}\subref{k=0}).
\end{remark}
To conclude the proof of Theorem \ref{THE_main}, we consider the lexicographic ordering $l:\mathbf{Z}^2\rightarrow \mathbf{N}$ and set $\mathcal{H}_{l(k,m)}:=\mathcal{H}_{k,m}$ with corresponding orthogonal projection $p_{l(k,m)}:=p_{(k,m)}$ for any $(k,m)\in\mathbf{Z}^2$; hence, \eqref{eq:rules} and Theorem \ref{thm:top}(a) imply that $\mathcal{H}:=\bigoplus_{(k,m)\in\mathcal{N}}\mathcal{H}_{k,m}$ and $\mathcal{G}:=\bigoplus_{(k,m)\in\mathbf{Z}^2\setminus \mathcal{N}}\mathcal{H}_{k,m}$ satisfy the statement (i) of Theorem \ref{THE_main}. Finally, let $\psi_0,\psi_1$ be in $\mathcal{H}$ and such that $\|p_n(\psi_0)\|=\|p_n(\psi_1)\|$ for all $n\in\mathbf{N}$. For $\epsilon>0$ let $r\in\mathbf{N}$ be such that $\|\psi_0-\bigoplus_{i=0}^rp_i(\psi_0)\|<\epsilon/3,
\|\psi_1-\bigoplus_{i=0}^rp_i(\psi_1)\|<\epsilon/3.$ By Theorem \ref{thm:top}(b), there exists $u\in L^1([0,T])$ such that $\|\Upsilon^u_T(\bigoplus_{i=0}^rp_i(\psi_0))-\bigoplus_{i=0}^rp_i(\psi_1)\|<\epsilon/3.$ By triangular inequality, we have that  $\|\Upsilon^u_T(\psi_0)-\psi_1\|<\epsilon.$

\begin{figure}[ht!]\begin{center}
\subfigure[]{
\includegraphics[width=0.48\linewidth, draft = false]{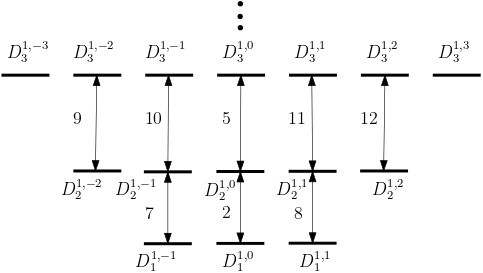} \label{k=1} }
\subfigure[]{
\includegraphics[width=0.48\linewidth, draft = false]{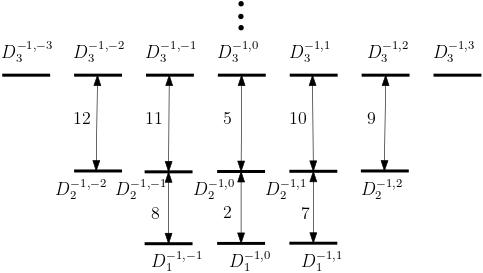} \label{k=-1} }
\subfigure[]{
\includegraphics[width=0.48\linewidth, draft = false]{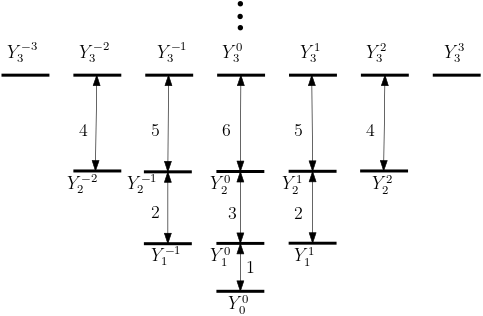} \label{k=0} }
\caption{$H_z$ acting on the spectral graphs of: \subref{k=1} $H_{\rm rot}|_{\bigoplus_{m\in\mathbf{Z}}\mathcal{H}_{1,m}}$; \subref{k=-1} $H_{\rm rot}|_{\bigoplus_{m\in\mathbf{Z}}\mathcal{H}_{-1,m}}$; \subref{k=0} $H_{\rm rot}|_{\bigoplus_{m\in\mathbf{Z}}\mathcal{H}_{0,m}}$. Arrows with same numbers correspond to related transitions; arrows with different numbers correspond to simultaneously controllable transitions.}\label{fig:graphs}
\end{center}
\end{figure}

\subsection{Proof of Proposition \ref{PRO_compelment}}
Since the restriction of $H_z$ is bounded from $D(H_{\rm rot}^k)$ to itself for every 
integer $k$ in $\mathbf{N}$, the system $(\ii H_{\rm rot},\ii H_z)$ is indeed $k$-midly coupled 
for every $k$ in the sense of Definition 5 in \cite{Chambrion-Caponigro-Boussaid-2020}. Proposition 
\ref{PRO_compelment} is a consequence of the density of polynomials in $L^1([0,T],\mathbf{R})$
for any $T>0$ and Proposition 23 in \cite{Chambrion-Caponigro-Boussaid-2020}.

\section{Numerical simulations of orientational selective transfer}\label{sec:simulations}
\subsection{Error estimate for finite-dimensional approximations}\label{sec:error}

In this section we formulate an estimate (which we use in Sec. \ref{sec:results}) of the error made by replacing the original system by one of its Galerkin approximations in the spirit of \cite{weaklycoupled}. 
\begin{definition}
The operator $B_m$ is said to be tri-diagonal w.r.t. $\Phi_m$ if, for any $j, j'\in \mathbf{N}$, $|j-j'|>1$ implies $\langle \phi^m_j,B_m\phi^m_{j'}\rangle=0$.
\end{definition}
Consider the orthogonal projection $\pi^m_N:\mathcal{H}_m\rightarrow \mathcal{H}_m^{(N)}:={\rm span}\{\phi^m_1,\dots,\phi^m_N\}$ on the first $N$ eigenfunctions of $A_m$ and
denote by $X^{u,m}_{(N)}(t,s)$ (for brevity $X^{u}_{(N)}(t,s)$ when it is applied to an initial datum in $\mathcal{H}_m^{(N)}$) the propagator of  
\begin{equation}\label{eq:mN}
\dot{x}=(A_m^{(N)}+uB_m^{(N)})x, \quad x\in\mathcal{H}_m^{(N)}\cong \mathbf{C}^N,
\end{equation}
where $A_m^{(N)}=\pi^m_NA_m\pi^m_N, B_m^{(N)}=\pi^m_NB_m\pi^m_N$. System \eqref{eq:mN} is usually called the $N$-dimensional Galerkin approximation of \eqref{eq:m}.
\begin{proposition}\label{prop:error}
Let $B_m$ be tri-diagonal w.r.t. $\Phi_m$. Then, for every $\psi_0\in\mathcal{H}_m$, $N_1,N\in\mathbf{N}$ with $N_1\leq N$ and $u\in L^1_{\rm loc}(\mathbf{R},\mathbf{R})$,
\begin{align}\label{eq:estimate}
\left\|\pi^m_{N_1} \Upsilon^u_t(\psi_0)-\pi^m_{N_1}X^u_{(N)}(t,0)\pi^m_N\psi_0 \right\| 
\leq \|u\|_{L^1([0,t])}|b_{(N,m),(N+1,m)}|\sup_{s\in[0,t]}\left\|\pi^m_{N_1}X^u_{(N)}(t,s)\phi_N^{m}\right\|. 
\end{align}
\end{proposition}
\begin{proof}
We have
\begin{align*}
\frac{d}{dt}\pi^m_N \Upsilon^u_t(\psi_0)=(A_m^{(N)}+u B_m^{(N)})\pi^m_N \Upsilon^u_t(\psi_0)+u \pi_N^{m}B_m({\rm Id}-\pi^m_N) \Upsilon^u_t(\psi_0).
\end{align*}
Using the variation of constants formula, we integrate
\begin{align*}
\pi^m_N \Upsilon^u_t(\psi_0)=X^u_{(N)}(t,0)\pi^m_N\psi_0+\int_0^tu(s)X^u_{(N)}(t,s)\pi_N^{m}B_m({\rm Id}-\pi^m_N) \Upsilon^u_s(\psi_0)ds, 
\end{align*}
and then we project on a subspace of dimension $N_1\leq N$
\begin{align*}
\pi^m_{N_1} \Upsilon^u_t(\psi_0)=\pi^m_{N_1}X^u_{(N)}(t,0)\pi^m_N\psi_0+\int_0^tu(s)\pi^m_{N_1}X^u_{(N)}(t,s)\pi_N^{m}B_m({\rm Id}-\pi^m_N) \Upsilon^u_s(\psi_0)ds.
\end{align*}
Thanks to the tri-diagonal structure, we have 
$$\pi_N^{m}\!B_m({\rm I}-\pi^m_N)\! \Upsilon^u_s(\psi_0)\!=\!b_{(N,m),(N+1,m)}\langle \phi^m_{N+1},\!\Upsilon^u_s(\psi_0)\rangle \phi^m_N,$$
and the thesis follows.
\end{proof}
\begin{remark}
For applications, $N$ and $N_1$ are the dimensions of the spaces, respectively, where the numerical simulation is performed and where the transfer approximately happens.
\end{remark}

\subsection{Construction of the control laws and results}\label{sec:results}

In this final section, considering $A=1$, $C=2$ in \eqref{EQ_def_Hrot},  we numerically simulate the transfer between the two rotational states 
\begin{align*}
\psi_0=\frac{1}{\sqrt{3}}(D_1^{1,-1}+D_1^{1,0}+D_1^{1,1}), \quad
 \psi_1=\frac{1}{\sqrt{3}}(D_1^{1,-1}+D_1^{1,0}+D_2^{1,1}).
 \end{align*}
 In the spirit of \cite{chambrion}, we consider the control function
\begin{align}\label{eq:control}
&u(t)=1+\frac{1}{25} \Big(2.38\sin (3.71\,t)-4.42\sin (9.63\,t)  \\ &+7.13\sin (17.59\,t)
+0.01\sin (5.91\,t)-0.02\sin (13.88\,t)\Big), \nonumber
\end{align}
which is a suitable linear combination of periodic functions that oscillate in resonance 
with the spectral gaps of the perturbed drift $(H_{\rm rot}+H_z)|_{\mathcal{H}_{1,1}}$ 
corresponding to $(k,m)=(1,1)$. Denoting by $R$ the matrix of the target rotation,
the coefficients in \eqref{eq:control} 
are obtained as the ratio between the off-diagonal entries of the matrices $\log 
R$ and $H_z$, both expressed in a basis where $H_{\rm rot}+H_z$ is diagonal. 

We use the control law \eqref{eq:control} on subspaces  spanned by the first $N=10$ energy levels of the spaces $\mathcal{H}_{1,m}$, $m=-1,0,1$  (with error less than $10^{-6}$ on the first $N_1=2$, by Prop. \ref{prop:error}). The results are presented on Fig. \ref{FIG_Transfert_reussi}. 
The Octave/Matlab script used for the computation is available on the companion webpage of this paper.

\begin{figure}
\centering
\includegraphics[width=0.68\textwidth,trim=120 250 130 340, clip]{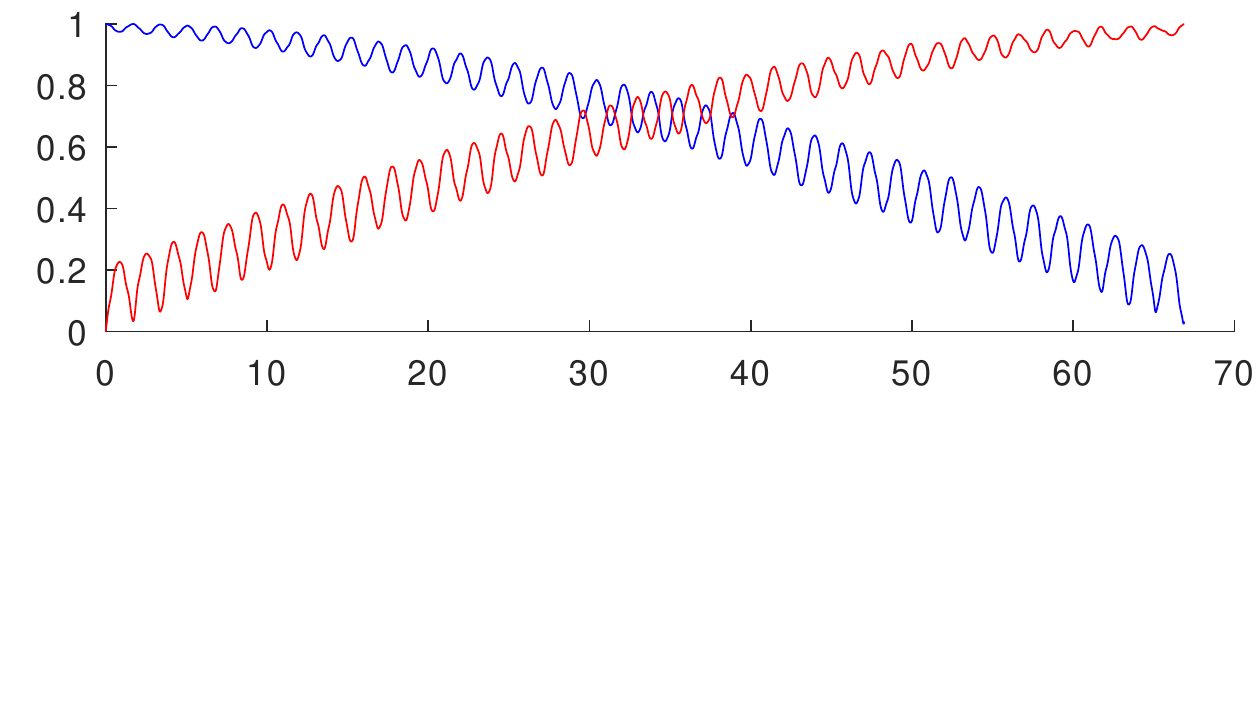}
\vspace{-2.2cm}

\includegraphics[width=0.68\textwidth,trim=120 360 130 370, clip]{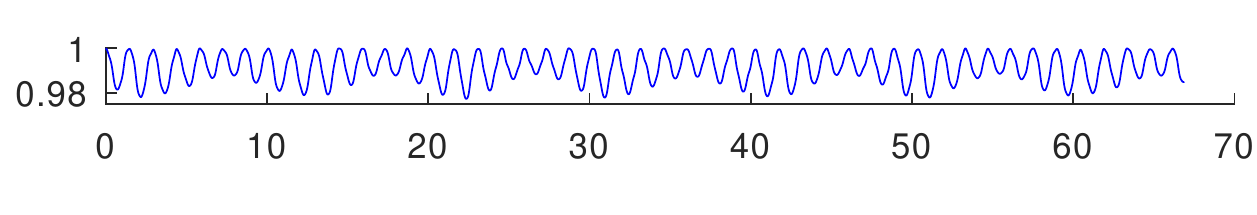}
\caption{Top. Evolution with respect to time of the components moduli $|\langle D_1^{1,1}, \Upsilon^u_t( D_1^{1,1})\rangle |$ (in blue) and $|\langle D_2^{1,1}, \Upsilon^u_t( D_1^{1,1})\rangle |$ (in red). The control law $u$ is given by \eqref{eq:control}.  The maximum  of the red curve $0.999$ is obtained at time $T=66.889$. \newline
Bottom. Evolution with respect to time of the component modulus $|\langle D_1^{1,-1}, \Upsilon^u_t( D_1^{1,-1})\rangle |$. At time $T=66.889$, $|\langle D_1^{1,-1}, \Upsilon^u_t( D_1^{1,-1})\rangle |>0.985$. The picture is similar for $|\langle D_1^{1,0}, \Upsilon^u_t( D_1^{1,0})\rangle |$.}\label{FIG_Transfert_reussi}
\end{figure}


\section{Conclusion}

We have exposed the controllability properties of the orientation of a symmetric molecule. 
While the result is constructive, further work is needed to optimize the choice of the parameters (especially the shift of the drift) in order to minimize the controllability time. \\

\textbf{Acknowledgments}\\
The authors thank C.P.Koch, M.Leibscher and D.Sugny for fruitful discussions.

This work is part of the project CONSTAT, supported by the Conseil 
Régional de Bourgogne Franche Comté and the European Union through the PO FEDER 
Bourgogne 2014/2020 programs,  by the French ANR through the grant QUACO (17-
CE40-0007-01) and by EIPHI Graduate School (ANR-17-EURE-0002).

\bibliographystyle{siamplain}
\bibliography{references}

\end{document}